\documentclass[a4paper,reqno]{amsart}
\usepackage{graphicx} % Required for inserting images
\usepackage{xcolor}

\usepackage[T1]{fontenc}
\usepackage[latin1]{inputenc}
\usepackage{amssymb}
\usepackage{amsfonts}
\usepackage{amsmath,mathtools}
\usepackage{mathrsfs}
\usepackage{graphicx}
\usepackage[notref, notcite]{showkeys}
\usepackage{showkeys}
\usepackage{hyperref}
\usepackage{esint}
\usepackage{verbatim}
\usepackage{bbm}
\usepackage{tikz}
\newcommand{\R}[1]{\mathbb{R}^{#1}}

\newcommand{\de}{\mathrm d}

\newcommand{\eps}{\varepsilon}

\renewcommand{\geq}{\geqslant}
\renewcommand{\leq}{\leqslant}

\newcommand{\wsto}{\stackrel{*}{\rightharpoonup}}

\newcommand{\average}{{\mathchoice {\kern1ex\vcenter{\hrule
height.4pt width 8pt depth0pt}
\kern-11pt} {\kern1ex\vcenter{\hrule height.4pt width 4.3pt
depth0pt} \kern-7pt} {} {} }}

\mathchardef\emptyset="001F

\providecommand{\U}[1]{\protect\rule{.1in}{.1in}}
\numberwithin{equation}{section}
\def\e{\eps}

\definecolor{vg}{rgb}{0.0, 0.45, 0.10}

\newtheorem{definition}{Definition}[section]
\newtheorem{theorem}[definition]{Theorem}
\newtheorem{lemma}[definition]{Lemma}
\newtheorem{proposition}[definition]{Proposition}
\newtheorem{remark}[definition]{Remark}

\title[SBH-ellipticity of the relaxed surface energy density for structured deformations]{SBH-ellipticity of the relaxed interfacial energy density in the context of second-order structured deformations}

\author[A.~C.~Barroso]{Ana Cristina Barroso}
\address[A.~C.~Barroso]{Departamento de Matem\'atica and CEMS.UL, 
	Faculdade de Ci\^encias da Universidade de Lisboa,
	Campo Grande, Edif\' \i cio C6, Piso 1,
	1749-016 Lisboa, Portugal}
\email{acbarroso@ciencias.ulisboa.pt}
\author[J.~Matias]{Jos\'{e} Matias}
\address[J.~Matias]{Departamento de Matem\'atica, Instituto Superior T\'ecnico, Av.~Rovisco Pais, 1, 1049-001 Lisboa, Portugal}
\email{jose.c.matias@tecnico.ulisboa.pt}

\author[E.~Zappale]{Elvira Zappale}
\address[E.~Zappale]{Dipartimento di Scienze di Base ed applicate per l'Ingegneria, Sapienza - Universit\`{a} di Roma, Via A. Scarpa, 16, 00161, Roma, Italy, and CIMA, Universidade de \'Evora, Portugal}
\email{elvira.zappale@uniroma1.it}

\makeindex

\begin{document}

\begin{abstract}

Starting from an energy comprised of both a bulk term and a surface term, set in the space of special functions of bounded hessian, $SBH$, a relaxation problem in the context of second-order structured deformations was studied in \cite{FHP}. It was shown, via the global method for relaxation, that the relaxed functional admits an integral representation and the relaxed energy densities were identified.
In this paper we show that, under certain hypotheses on the original densities, the corresponding relaxed energy densities verify the same type of growth conditions and the surface energy density satisfies a specific ``convexity-type'' property, i.e. it  is $SBH$-elliptic.

\bigskip

	%46E30, 74A60, 74M99, 74B20


\color{black}
\end{abstract}

\maketitle

\allowdisplaybreaks

\tableofcontents

\color{black}
\section{Introduction}

\color{black}

Structured deformations were introduced by Del Piero and Owen in \cite{DPO1993}  in order to provide a multi-scale geo\-me\-tri\-cal setting to describe the fields occurring in continuum mechanics, decomposing them into two parts: one arising from smooth changes at smaller length scales and another part due to disarrangements (such as slips and separations) occurring at smaller length scales. The portion generated by the part of deformation without disarrangements is computed as a limit (in an appropriate sense) of gradients of approximating deformations, and the portion generated by disarrangements corresponds to the effects of jumps in the approximating deformations. Adopting the same  point of view as in \cite{DPO1993},  Owen and Paroni in \cite{OP2000} extended the principal concepts and results of the above theory to include the effects of limits of second gradients and jumps in the first gradients of approximating deformations. In particular, this enriched theory allows to take into account bending and torsion effects.

In \cite{FHP}, 
the authors address second-order structured deformations, \color{black} in the context of special functions of bounded Hessian, $SBH$,  with the main aim of describing  their  energetic behaviour (we refer to Section \ref{prelim} for more details and for the rigorous framework of first-order structured deformations and their second-order variants). In particular, in \cite{FHP} the energy assigned to a second-order structured deformation was defined and computed, in terms of relaxation, departing from a hyper-elastic energy that also depends on second-order distributional gradients.
 Indeed, an integral representation result for a quite general class of energy  functionals satisfying standard properties (see Section \ref{mr} for more details) was derived, in the spirit of the global method for relaxation, introduced by Bouchitt\'{e}, Fonseca and Mascarenhas in \cite{BFM}, and later extended in \cite{BFLM}. 
 The limiting energy,  obtained  via the above-mentioned relaxation procedure, is described in terms of suitable densities, given by means of certain  formulas (see Section \ref{mr} for the precise  expressions), but no properties enjoyed by these densities have been discussed.  

In this work, departing from a specific model, as a particular case of some applications considered in \cite{FHP}, and having in mind the characterisation of the relaxed surface energy in the context of (first-order) structured deformations obtained by \v Silhav\'y  in \cite{Silhavy},
we prove the $SBH$-ellipticity of the corresponding relaxed surface energy density.

More precisely, we consider, over an open and bounded set $\Omega \subset \mathbb R^N$, the family of second-order structured deformations
$$ SD_2(\Omega) := SBH( \Omega; \mathbb R^d) \times 
L^1( \Omega; \mathbb R_{sym}^{d\times N\times N})$$
where $\mathbb R_{sym}^{d \times N\times N} \subset \mathbb R ^{d\times N\times N}$  denotes the set of tensors 
$(M_{ijk}), i \in \{1, \ldots d\}, j, k, \in \{1, \ldots N\},$ such that 
$M_{ijk} = M_{ikj}$ for all $i \in \{1, \ldots d\}, j, k \in \{1, \dots N\}.$

Let $\mathcal{O}(\Omega)$ denote  the family of open subsets of $\Omega$ and let 
$F_0$  be the functional defined for each 
$ A \in \mathcal{O}(\Omega)$ by
\begin{equation}\label{Funct}
F_0(u;A) := \begin{cases}
\displaystyle  \int_A f_0(\nabla^2u(x))\, dx + 
\int_A f_1(x,u(x),\nabla u(x)) \, dx & \\
\hspace{1,5cm}+
\displaystyle \int_{S(\nabla u)\cap A} g_0([\nabla u](x),\nu_{\nabla u}(x)) \, d\mathcal H^{N-1}(x),  & \mbox{if } u \in SBH(\Omega; \mathbb R^d) \\
+\infty,     & \text{otherwise},
\end{cases}
\end{equation}

where the densities $f_0$, $f_1$ and $g_0$ satisfy the following hypotheses:

\begin{enumerate}
    \item[(G1)] $f_0 : \mathbb R^{d\times N\times N}_{sym} \to [0, +\infty[$ is continuous and satisfies 
    $$ \frac{1}{C_1} |\Lambda| \leq f_0( \Lambda) \leq C_1( |\Lambda|),$$
    for all $\Lambda \in \mathbb R^{d\times N\times N}_{sym}$ and for some $C_1>0;$
    \item[(G2)] $g_0:\mathbb R^{d\times N} \times \mathbb{S}^{N-1} \color{black}\to [0, +\infty[$ is continuous and  satisfies 
   $$ \frac{1}{C_2} |\xi| \leq g_0( \xi,\nu) 
    \leq C_2|\xi|,$$
    as well as the symmetry condition
    $$g_0(\xi, \nu) = g_0(-\xi,-\nu),$$
    for all $(\xi,\nu) \in \mathbb R^{d\times N}\times \mathbb{S}^{N-1}$ and for some $C_2> 0$;
\item[(G3)] $f_1:\Omega \times \mathbb R^d \times \mathbb R^{d\times N}\to [0,+\infty)$ is a Carath\'eodory function, i.e. it is $\mathcal L^N$-measurable in the first variable and continuous in the remaining ones, such that
$$
\frac{1}{C'}|\xi|^q-C'\leq f_1(x,a, \xi)\leq C'(1+|\xi|^q),
$$
for a.e. $x \in \Omega$ and every $(a,\xi)\in \mathbb R^d \times \mathbb R^{d\times N}$, and for some $C'>0$ and $1< q \leq \frac{N}{N-1}$.
\end{enumerate}
 Although we do not use this property in the proofs, from the mechanical point of view, it is sometimes required that $g_0$ be sub-additive in the ${\xi}$ variable. 
The functional $F_0$ in \eqref{Funct} can be interpreted as the energy associated to classical deformations in multi-phase materials, in fact, it is the sum of a classical hyper-elastic energy and a penalisation term intended to detect the effects of meso-level phase transitions, since it penalises the variation of the deformation gradient, both in its smooth changes and in its jumps.

Consider the functional 
$\mathcal F : SD_2(\Omega) \times \mathcal{O}(\Omega) \to [0, +\infty[$ defined by
\begin{equation}\label{relaxf}
\mathcal{F}(u,U;A) := \inf \left\{ \liminf_{n\to +\infty} F_0(u_n; A) : u_n \to u \; \text{in} \; W^{1,1}(\Omega; \mathbb R^d), \color{black}\nabla^2 u_n \wsto U \; \text{in}\; \mathcal{M} (\Omega; \mathbb R^{d\times N\times N}_{s})\right\}.
\end{equation}

 It is well known that, under the above assumptions, the term $\displaystyle \int_A f_1(x,u(x), \nabla u(x)) \, dx$, present in $F_0$, is con\-ti\-nu\-ous with respect to the $W^{1,1}(\Omega;\mathbb R^d)$ convergence of $u_n$ to $u$ (see \cite{DM93}), hence, when computing the right hand side of \eqref{relaxf}, that term can be neglected and added after the relaxation procedure, i.e. 
\begin{align*}
\mathcal F(u,U;A) = \inf \left\{ \liminf_{n\to +\infty} 
\left(\int_A f_0(\nabla^2 u_n(x)) \, dx +
\int_{S(\nabla u_n) \cap A} g_0([\nabla u_n](x),\nu_{\nabla u_n}(x)) 
\, d\mathcal H^{N-1}(x)\right) :\right.\\
\left. u_n \to u \; \text{in} \; W^{1,1}(\Omega; \mathbb R^d), 
\nabla^2 u_n \wsto U \; \text{in}\; 
\mathcal{M} (\Omega; \mathbb R_{s}^{d\times N\times N})\right\}
+\int_A f_1(x, u(x), \nabla u(x)) \, dx.
\end{align*}

We also point out that, due to the interpolation inequality stated in Theorem \ref{interpolBH}, one can replace the $W^{1,1}(\Omega;\mathbb R^d)$ convergence of $u_n$ to $u$ in \eqref{relaxf}, by the $L^{1}(\Omega;\mathbb R^d)$ convergence of $u_n$ to $u$, obtaining in both cases the same result.

An integral representation for $\mathcal {F}$, for a more general form of the energy $F_0$, was obtained in \cite[Theorem 5.4]{FHP}, through the application of the global method for relaxation results derived in the same article
(see also \cite{BMZ2024} and \cite{BZ2025} for related problems).

In particular, from those results, it can be shown that, under hypotheses (G1)-(G3), there exist functions 
$f: \mathbb R^{d\times N\times N}_{sym}\times \mathbb R^{d\times N\times N}_{sym} 
\to [0, +\infty[$ and 
$g: \mathbb R^{d\times N} \times \mathbb {S}^{N-1} 
\to [0, +\infty[ $ such that
\begin{align*}
\mathcal{F}(u,U; A) &= \int_A f(\nabla^2 u(x), U(x))\, dx + 
\int_A f_1(x, u(x), \nabla u(x)) \,dx \\
&+
\int_{S(\nabla u)\cap A} g([\nabla u](x),\nu_{\nabla u}(x))
\, d \mathcal H^{N-1}(x),
\end{align*}
for all $(u,U) \in SD_2(\Omega)$ and all $A \in \mathcal{O}(\Omega)$.
Furthermore, the relaxed energy densities $f$ and $g$ can be identified in terms of a related local Dirichlet-type functional (see Section \ref{mr} for more details).
Our aim in this paper is to show that the function $g$ is $SBH$-elliptic in the sense of Definition \ref{BHelldef} below.

We organise the paper as follows. In Section \ref{prelim} we set the notation and state some results on $BV$ and $BH$ fields which will be used in the sequel. We also recall the notions of first-order and second-order structured deformations, as well as the concept of $SBH$-ellipticity. In Section \ref{mr} we give a brief overview of some of the results obtained in \cite{FHP} that pertain to our problem and we state and prove our main result on the $SBH$-ellipticity of the relaxed surface energy density.

\color{black}

\section{Preliminaries}\label{prelim}

\subsection{Notation} 

We will use the following notations
\begin{itemize}
	\item $\mathbb N$ denotes the set of natural numbers without the zero element;
	\item $\Omega \subset \mathbb R^{N}$ is a bounded, connected open set with Lipschitz boundary; 
	\item $\mathbb S^{N-1}$ denotes the unit sphere in $\mathbb R^N$;
	\item $Q\coloneqq (-\tfrac12,\tfrac12)^N$ denotes the open unit cube of $\mathbb R^{N}$ centred at the origin; for any $\nu\in\mathbb S^{N-1}$, $Q_\nu$ denotes any open unit cube in $\mathbb R^{N}$ with two faces orthogonal to $\nu$; 
	\item for any $x\in\mathbb R^{N}$ and $\delta>0$, $Q(x,\delta)\coloneqq x+\delta Q$ denotes the open cube in $\mathbb R^{N}$ centred at $x$ with side length $\delta$; likewise $Q_\nu(x,\delta)\coloneqq x+\delta Q_\nu$;
	\item ${\mathcal O}(\Omega)$ is the family of all open subsets of $\Omega $, whereas ${\mathcal O}_\infty(\Omega)$ is the family of all open subsets of $\Omega $ with Lipschitz boundary; 
	\item $\mathcal L^{N}$ and $\mathcal H^{N-1}$ denote the  $N$-dimensional Lebesgue measure and the $\left(  N-1\right)$-dimensional Hausdorff measure in $\mathbb R^N$, respectively; the symbol $\de x$ will also be used to denote integration with respect to $\mathcal L^{N}$; 
	\item $\mathcal M(\Omega;\mathbb R^{d\times N})$ is the set of finite matrix-valued Radon measures on $\Omega$; $\mathcal M ^+(\Omega)$ is the set of non-negative finite Radon measures on $\Omega$;
	given $\mu\in\mathcal M(\Omega;\mathbb R^{d\times N})$,  
	the measure $|\mu|\in\mathcal M^+(\Omega)$ 
	denotes the total variation of $\mu$;
	\item $L^p(\Omega;\mathbb R^{d\times N})$ is the set of matrix-valued $p$-integrable functions; 
	\item  $C$ represents a generic positive constant that may change from line to line.
\end{itemize}

\medskip

\subsection{Fields of Bounded Variation}\label{BV}

In the following we give some preliminary notions regarding functions of bounded variation and sets of finite perimeter. For a detailed treatment we refer to \cite{AFP}.

\smallskip

Given $u \in L^1(\Omega; \mathbb R^d)$ we let $\Omega_u$ be the set of Lebesgue points of $u$,
i.e., $x\in \Omega_u$ if there exists $\widetilde u(x)\in {\mathbb{R}}^d$ such
that 
\begin{equation*}
	\lim_{\varepsilon\to 0^+} \frac{1}{\e^N}\int_{B(x,\varepsilon)}
	|u(y)-\widetilde u(x)|\,dy=0,
\end{equation*}
$\widetilde u(x)$ is called the approximate limit of $u$ at $x$.
The Lebesgue discontinuity set $S_u$ of $u$ is defined as 
$S_u := \Omega \setminus \Omega_u$. It is known that ${\mathcal{L}}^{N}(S_u) = 0$ and the function 
$x \in \Omega \mapsto \widetilde u(x)$, which coincides with $u$ $\mathcal L ^N$- a.e.
in $\Omega_u$, is called the Lebesgue representative of $u$.

The jump set of the function $u$, denoted by $J_u$, is the set of
points $x\in \Omega \setminus \Omega_u$ for which there exist 
$a, \,b\in {\mathbb{R}}^d$ and a unit vector $\nu \in \mathbb S^{N-1}$, normal to $J_u$ at $x$, such that $a\neq b$ and 
\begin{equation*}  
	\lim_{\varepsilon \to 0^+} \frac {1}{\varepsilon^N} \int_{\{ y \in
		B(x,\varepsilon) : (y-x)\cdot\nu > 0 \}} | u(y) - a| \, dy = 0,
	\qquad
	\lim_{\varepsilon \to 0^+} \frac {1}{\varepsilon^N} \int_{\{ y \in
		B(x,\varepsilon) : (y-x)\cdot\nu < 0 \}} | u(y) - b| \, dy = 0.
\end{equation*}
The triple $(a,b,\nu)$ is uniquely determined by the conditions above,  
up to a permutation of $(a,b)$ and a change of sign of $\nu$,
and is denoted by $(u^+ (x),u^- (x),\nu_u (x)).$ The jump of $u$ at $x$ is defined by
$[u](x) : = u^+(x) - u^-(x).$

\smallskip

We recall that a function $u\in L^{1}(\Omega;{\mathbb{R}}^{d})$ is said to be of bounded variation, and we write $u\in BV(\Omega;{\mathbb{R}}^{d})$, if
all its first order distributional derivatives $D_{j}u_{i}$ belong to 
$\mathcal{M}(\Omega)$ for $1\leq i\leq d$ and $1\leq j\leq N$.

The matrix-valued measure whose entries are $D_{j}u_{i}$ is denoted by $Du$
and $|Du|$ stands for its total variation.
The space $BV(\Omega; {\mathbb{R}}^d)$ is a Banach space when endowed with the norm 
\begin{equation*}
	\|u\|_{BV(\Omega; {\mathbb{R}}^d)} = \|u\|_{L^1(\Omega ; {\mathbb{R}}^d)} + |Du|(\Omega )
\end{equation*}
and we observe that if $u\in BV(\Omega;\mathbb{R}^{d})$ then $u\mapsto|Du|(\Omega)$ is lower semi-continuous in $BV(\Omega;\mathbb{R}^{d})$ with respect to the
$L_{\mathrm{loc}}^{1}(\Omega;\mathbb{R}^{d})$ topology.

By the Lebesgue Decomposition Theorem, $Du$ can be split into the sum of two
mutually singular measures $D^{a}u$ and $D^{s}u$, the absolutely continuous
part and the singular part, respectively, of $Du$ with respect to the
Lebesgue measure $\mathcal{L}^N$. By $\nabla u$ we denote the 
Radon-Nikod\'{y}m derivative of $D^{a}u$ with respect to $\mathcal{L}^N$, so that we
can write 
\begin{equation*}
	Du= \nabla u \mathcal{L}^N \lfloor \Omega + D^{s}u.
\end{equation*}

If $u \in BV(\Omega )$ it is well known that $S_u$ is countably $(N-1)$-rectifiable, see \cite{AFP},  
and the following decomposition holds 
\begin{equation*}
	Du= \nabla u \mathcal{L}^N \lfloor \Omega + [u] \otimes \nu_u {\mathcal{H}}^{N-1}\lfloor S_u + D^cu,
\end{equation*}
\noindent where $D^cu$ is the Cantor part of the
measure $Du$.  When $D^cu = 0$, the function $u$ is said to be a special function of bounded variation, written
$u \in SBV(\Omega;\mathbb R^d)$. 

If $\Omega$ is an open and bounded set with Lipschitz boundary then the
outer unit normal to $\partial \Omega$ (denoted by $\nu$) exists ${\mathcal{H}}^{N-1}$-a.e. and the trace for functions in $BV(\Omega;{\mathbb{R}}^d)$ is
defined.

\subsection{Fields of Bounded Hessian}\label{BH}

The space of fields with bounded hessian, introduced in \cite{D, DT} and denoted by $BH(\Omega;\mathbb R^d)$, is defined as
\begin{align*}
BH(\Omega;\mathbb R^d):=&\left\{u \in W^{1,1}(\Omega;\mathbb R^d): D(\nabla u) \hbox{ is a bounded Radon measure }\right\} \\
=& \left\{u \in L^1(\Omega;\mathbb R^d): 
Du \in BV(\Omega;\mathbb R^{d\times N})\right\},
\end{align*}
the measure $D(\nabla u)$ is often written as $D^2u$.
The space $BH(\Omega;\mathbb R^d)$ is endowed with the norm
\begin{align*}
\|u\|_{BH(\Omega;\mathbb R^d)} :=\|u\|_{W^{1,1}(\Omega;\mathbb R^d)}
+ |D^2 u|(\Omega),
\end{align*}
where the latter term denotes the total variation of the Radon measure $D^2 u$.

Given $u \in BH(\Omega;\mathbb R^d)$, 
taking into account that $Du=\nabla u \in BV(\Omega;\mathbb R^d)$, the following decomposition holds, 
\begin{equation*}
D(\nabla u)= \nabla^2 u \mathcal L^N + D^s(\nabla u) =
\nabla^2 u \mathcal L^N + 
[\nabla u] \otimes \nu_{\nabla u} \mathcal H^{N-1}\lfloor S_{\nabla u} 
+ D^c(\nabla u).
\end{equation*}

Due to the symmetry of the distribution $D^2 u = D(\nabla u)$, in the previous expression,  $\nabla^2 u \mathcal L^N$ is a vector-valued measure, in the space of 
$d\times N \times N$  symmetric tensors $\mathbb R^{d \times N\times N}_{sym}$, that is absolutely continuous with respect to $\mathcal L^N$ and with 
$$\displaystyle \nabla^2 u= \frac{\partial ^2 u_i}{\partial x_j \partial x_k}= \frac{\partial^2 u_i}{\partial x_k \partial x_j}, \; 1 \leq i \leq d, \; 1 \leq j,k \leq N,$$ 
$S_{\nabla u}$ denotes the singular set of $\nabla u$, $\nu_{\nabla u}$ is the normal to $S_{\nabla u}$ , $[\nabla u]:= (\nabla u)^+- (\nabla u)^-$
is the jump across $S_{\nabla u}$ and $D^c(\nabla u)$ is the Cantor part of $D(\nabla u)$, which is singular with respect to  $\mathcal L^N \lfloor \Omega+ \mathcal H^{N-1}\lfloor S_{\nabla u}$.
 
Again by the symmetry of $D^2u$, it follows that $[\nabla u] = \alpha  \otimes \nu_{\nabla u}$, for some function $\alpha$ which is integrable in $\Omega$ with respect to the measure $\mathcal H^{N-1}\lfloor S_{\nabla u}$, due to the fact that
$[\nabla u] \otimes \nu_{\nabla u} $ is a $d$-tuple of $N \times N$ symmetric rank-one matrices.

The space $SBH(\Omega;\mathbb R^d)$ consists of those functions $u \in BH(\Omega;\mathbb R^d)$ for which $D^c(\nabla u) = 0$,
that is,
$$SBH(\Omega;\mathbb R^d):= \left\{u \in L^1(\Omega;\mathbb R^d): 
Du \in SBV(\Omega;\mathbb R^{d\times N})\right\}.$$

The proof of the following interpolation inequality can be found in \cite[Theorem 2.2]{FLP2}.

\begin{theorem}\label{interpolBH}[Interpolation inequality]
	Let $\Omega \subset \mathbb R^N$ be an open, bounded set with Lipschitz boundary. 
	Then, for every $\varepsilon > 0$, there is a constant $C = C(\varepsilon)$ such that
	$$\|\nabla u\|_{L^1(\Omega;\mathbb R^{d \times N})} \leq C\|u\|_{L^1(\Omega;\mathbb R^d)} + \varepsilon|D^2u|(\Omega),$$
	for all $u \in BH(\Omega;\mathbb R^d)$.
\end{theorem}

\color{black}

\subsection{Structured Deformations}
First-order structured deformations were introduced by Del Piero and Owen \cite{DPO1993} in order to provide a mathematical framework that captures the effects at the macroscopic level of smooth deformations and of non\--\-smooth deformations (disarrangements) at one sub-macroscopic level. 
In the classical theory of mechanics, the deformation of the body is characterised exclusively by the ma\-cros\-co\-pic deformation field~$g$ and its gradient~$\nabla g$. 
In the framework of structured deformations, an additional geometrical field~$G$ is introduced with the intent to capture the contribution at the macroscopic scale of smooth sub-macroscopic changes, while the difference $\nabla g-G$ captures the contribution at the macroscopic scale of non-smooth sub-macroscopic changes, such as slips and separations (referred to as \emph{disarrangements} in \cite{DO2002}). \color{black} The field~$G$ is called the deformation without disarrangements, and, heuristically, the disarrangement tensor $M\coloneqq \nabla g-G$ is an indication of how ``non-classical'' a structured deformation is.
This broad theory is rich enough to address mechanical phenomena such as elasticity, plasticity and the behaviour of crystals with defects.

The variational formulation for first-order structured deformations in the $SBV$ setting was first addressed by Choksi and Fonseca \cite{CF1997} where a (first-order) structured deformation is defined to be a pair 
$(g,G)\in SBV(\Omega;\R{d})\times L^p(\Omega;\R{d\times N}), \; p \geq 1.$
Departing from 
a functional which associates to any deformation $u$ of the body an energy featuring a bulk contribution, which measures the deformation (gradient) throughout the whole body, and an interfacial contribution, accounting for the energy needed to fracture the body, an integral representation for the ``most economical way'' to approach a given structured deformation was derived. 
The theory of first-order structured deformations was broadened by Owen and Paroni \cite{OP2000} to second\--\-order structured deformations, which also account for other geometrical changes, such as curvature, at one sub-macroscopic level. The variational formulation in the $SBV^2$ setting for second-order structured deformations was carried out by Barroso, Matias, Morandotti and Owen \cite{BMMO2017}. This formulation allows for jumps on both the approximating fields, as well as on their gradients. A formulation in $BH$, allowing  for jumps  only  in the gradients was carried out by Fonseca, Hagerty and Paroni \cite {FHP}.

\subsection{BV-ellipticity and BH-ellipticity}

With the aim of understanding the convexity-type property of our surface energy density obtained in Theorem \ref{prop4.12},
and in order to establish a 
comparison with the analogous property obtained in \cite{Silhavy} for the surface energy density in the case of first-order structured deformations,  we recall the notion of $BV$-ellipticity, given in \cite{AB}, and that of 
$SBH$-ellipticity defined in \cite{SZ2}.

\begin{definition}\label{BVelldef}
Let $T \subset \mathbb R^d$ be a finite set. A function
$\Phi: T \times T \times \mathbb S^{N-1} \to [0,+\infty[$ is said to be
$BV$-elliptic if the following inequality holds
$$\Phi(a,b,\nu) \leq \int_{Q_\nu \cap S_w}\Phi(w^+,w^-,\nu_\eta) \, d \mathcal H^{N-1},$$
for every $a, b \in T$, $\nu \in \mathbb S^{N-1}$ and for every piecewise constant function $w \in SBV(Q_\nu;T)$ such that $w=a$ in 
$\partial Q_\nu \cap \{x \cdot \nu > 0\}$
and
$w=b$ in $\partial Q_\nu \cap \{x \cdot \nu < 0\}$.
\end{definition}

Observe that when $T=\{a,b\}$, then Definition \ref{BVelldef} coincides with  $BV$-ellipticity with respect to $SJ_0$, as recently introduced in 
\cite[Definition 3.1]{EKM}. 
On the other hand, when dealing with fields which are a single pure jump, 
%that is, when 
%$a, b \in \{0,1\}$, 
among all the notions considered in
\cite[Definition 3.1]{EKM},  this is the only meaningful one. 

\begin{remark}\label{Bvellch}
In view of \cite[Theorem 5.14]{AFP} and standard arguments in the Calculus of Variations which, in $BV$, allow us to prescribe the boundary conditions of the approximating sequences so as to coincide with those of the target function, 
$BV$-ellipticity can be characterised as follows.

A function
$\Phi: T \times T \times \mathbb S^{N-1} \to [0,+\infty[$ is 
$BV$-elliptic if
%\begin{equation}\label{seqBVell}
$$\Phi(a,b,\nu) \leq \liminf_n\int_{Q_\nu \cap S_{w_n}}
\Phi(w_n^+,w_n^-,\nu_{w_n}) \, d \mathcal H^{N-1},
$$
%\end{equation}
for every $a, b \in T$, $\nu \in \mathbb S^{N-1}$ and for every sequence $\{w_n\} \subset SBV(Q_\nu; T)$ such that $w_n \to w_{a,b,\nu}$ in 
$L^{1}(Q_\nu;\mathbb R^d)$, $w_n = w_{a,b,\nu}$ on $\partial Q_\nu$, 
where $w_{a,b,\nu}$
is such that  $w=a$ in $\partial Q_\nu \cap \{x \cdot \nu > 0\}$
and $w=b$ in $\partial Q_\nu \cap \{x \cdot \nu < 0\}$.
\end{remark}

\begin{definition}\label{BHelldef}
Let $\mathcal P := \left \{(\xi,\eta,\nu) \in 
\mathbb R^{d \times N} \times \mathbb R^{d \times N} \times \mathbb S^{N-1} : 
\xi - \eta = \zeta \otimes \nu, \mbox{ for some } \zeta \in \mathbb R^d\right \}$. 
A continuous function
$\varphi: \mathcal P \to [0,+\infty[$ is said to be
$SBH$-elliptic if the following inequality holds
%\begin{equation}\label{seqSBHell}
$$\varphi(\xi,\eta,\nu) \leq \liminf_n\int_{Q_\nu \cap S_{\nabla u_n}}
\varphi((\nabla u_n)^+,(\nabla u_n)^-,\nu_{\nabla u_n}) \, d \mathcal H^{N-1},
$$
%\end{equation}
for every $(\xi,\eta,\nu) \in \mathcal P$ and for every sequence $\{u_n\} \subset SBH(Q_\nu;\mathbb R^d)$ such that 
$u_n \to u_{\xi,\eta,\nu}$ in 
$W^{1,1}(Q_\nu;\mathbb R^d)$, $u_n = u_{\xi,\eta,\nu}$\color{black} on $\partial Q_\nu$ and
$|\nabla^2 u_n| \to 0$ in $L^{1}(Q_\nu)$, where $u_{\xi,\eta,\nu}$\color{black} is given by 
$u_{\xi,\eta,\nu}(y) := \begin{cases}
\xi \cdot y, & {\rm if } \; y \cdot \nu \geq 0\\
\eta \cdot y, & {\rm if } \; y \cdot \nu < 0.
\end{cases}$
\end{definition}

\section{Main Result}\label{mr}

We begin this section by recalling some results proved in \cite{FHP} that are relevant to the problem under consideration in this paper.

Let $F$ be the functional defined for each $A \in \mathcal{O}(\Omega)$ by
\begin{equation}\label{FHPfunct}
F(u;A) := \begin{cases}
\displaystyle \int_A h_0(x,u(x),\nabla u(x),\nabla^2u(x)) \, dx &\\
\hspace{0,7cm} + \displaystyle \int_{S(\nabla u)\cap A} \psi_0(x,u(x),\nabla u^+(x), \nabla u^-(x), \nu_{\nabla u}(x)) \, d\mathcal H^{N-1}(x),  & \mbox{if } u \in SBH(\Omega; \mathbb R^d) \\
+\infty,     & \text{otherwise},
\end{cases}
\end{equation}
where the densities $h_0$ and $\psi_0$ satisfy the following hypotheses:
\begin{enumerate}
    \item[(H1)] $h_0 : \Omega \times \mathbb R^d \times \mathbb R^{d\times N}
    \times \mathbb R^{d\times N\times N} \to [0, +\infty[$ is measurable in the first variable and continuous in the remaining ones, and satisfies 
    $$ \frac{1}{C_1} |\Lambda| \leq h_0(x,a,\xi, \Lambda) \leq C_1(1 + |\Lambda|),$$
    for a.e. $x \in \Omega$ and every 
    $(a,\xi,\Lambda)\in \mathbb R^d \times \mathbb R^{d\times N} \times 
    \mathbb R^{d\times N\times N}$ and for some $C_1>0;$
    \item[(H2)] $\psi_0: \Omega \times \mathbb R^d \times (\mathbb R^{d\times N})^2 \times \mathbb{S}^{N-1} \to [0, +\infty[$ is continuous and  satisfies 
    $$ \frac{1}{C_2} |\xi - \eta| \leq \psi_0(x,a, \xi, \eta, \nu) 
    \leq C_2(1 + |\xi - \eta|),$$
    for all $(x,a,\xi, \eta, \nu) \in \Omega \times \mathbb R^d \times (\mathbb R^{d\times N})^2 \times \mathbb{S}^{N-1}$ and for some $C_2> 0$.
\end{enumerate}

Consider the functional 
$\mathcal G : SD_2(\Omega) \times \mathcal{O}(\Omega) \to [0, +\infty[$ defined by
\begin{equation}\label{FHPrelaxf}
\mathcal{G}(u,U;A) := \inf \left\{ \liminf_{n\to +\infty} F(u_n; A) : u_n \to u \; \text{in} \; L^1(\Omega; \mathbb R^d), \nabla^2 u_n \wsto U \; \text{in}\; \mathcal{M} (\Omega; \mathbb R^{d\times N\times N}_{sym})\right\}.
\end{equation}

It was shown in \cite[Lemmas 5.1 - 5.3]{FHP}, that $\mathcal G$ satisfies the following properties.

\begin{proposition}\label{Prop3.1} Let $h_0$ and $\psi_0$ be functions satisfying hypotheses (H1)-(H2) and let $\mathcal G$ be given by \eqref{FHPrelaxf}, where $F$ is defined in \eqref{FHPfunct}. Then
\begin{itemize}
\item[(i)] (lower semi-continuity) for every $(u,U) \in SD_2(\Omega)$, 
$A \in \mathcal O(\Omega)$ and every sequence $\{(u_n,U_n)\} \subset SD_2(\Omega)$ such that $u_n \to u$ in $L^1(\Omega;\mathbb R^d)$ and $U_n \wsto U$ in 
$\mathcal{M} (\Omega; \mathbb R_{sym}^{d\times N\times N})$, it follows that
$$\mathcal G(u,U;A) \leq \liminf_{n\to +\infty}\mathcal G(u_n,U_n;A);$$
\item[(ii)] (locality)  $\mathcal G$ is local, that is, for every
$A\in \mathcal O(\Omega)$, if $(u,U), (v,V) \in SD_2(\Omega)$ are such that $u=v$, $U=V$, for $\mathcal L^N$ a.e $x \in A$, then 
$\mathcal G(u,U;A) = \mathcal G(v,V;A)$;
\item[(iii)] (growth) there exists $C >0$ such that, for every $(u,U) \in SD_2(\Omega)$, $A \in \mathcal O(\Omega)$, we have
$$\frac{1}{C} \Big(\|U\|_{L^1(A;\mathbb R_{sym}^{d\times N\times N})} + |D^2u|(A)\Big) \leq \mathcal G(u,U;A) \leq C \Big(\mathcal L^N(A) + 
\|U\|_{L^1(A;\mathbb R_{sym}^{d\times N\times N})} + |D^2u|(A)\Big);$$
\item[(iv)] (Radon measure) for every $(u,U) \in SD_2(\Omega)$, the set function
$\mathcal G(u,U;\cdot)$ is the restriction to $\mathcal O(\Omega)$ of a Radon measure.
\end{itemize}
\end{proposition} 

The properties in the previous proposition ensure that $\mathcal G$ satisfies the hypotheses of \cite[Theorem 4.6]{FHP}  and hence
admits an integral representation of the form
\begin{align*}
\mathcal{G}(u,U; A) &= 
\int_A h_{\mathcal G}(x,u(x),\nabla u(x),\nabla^2 u(x), U(x))\, dx \\
& + \int_{S(\nabla u)\cap A} \psi_{\mathcal G}(x,u(x),(\nabla u)^+(x),
(\nabla u)^-(x),\nu_{\nabla u}(x))\, d \mathcal H^{N-1}(x),
\end{align*}
for all $(u,U) \in SD_2(\Omega)$ and all $A \in \mathcal{O}(\Omega)$.

Furthermore, taking into account that the proof of \cite[Theorem 4.6]{FHP} is based on the global method for relaxation introduced in \cite{BFM}, the relaxed energy density functions 
$h_{\mathcal G}: \Omega \times \mathbb R^d \times \mathbb R^{d\times N}\times \mathbb R^{d\times N\times N}_{sym} \times \mathbb R^{d\times N\times N}_{sym} 
\to [0, +\infty[$ and 
$\psi_{\mathcal G}:\Omega \times \mathbb R^d \times \mathbb R^{d\times N}\times \mathbb R^{d\times N} \times \mathbb {S}^{N-1} \to [0, +\infty[ $ are defined in the following way.

Given $(u,U;A) \in SD_2(\Omega) \times \mathcal O(\Omega)$, we consider the set 
\begin{equation}\label{adm}
\mathcal U(u,U;A) := \left\{(v,V) \in SD_2(\Omega) : \mbox{spt}(u-v) \subset \subset A, \int_A (U-V ) dx= 0\right\}
\end{equation}
and we define the functional
%\begin{equation}\label{FHPmfun}
$$m_{\mathcal G}(u,U;A) := \inf \big\{\mathcal G(v,V;A) : (v,V) \in \mathcal U(u,U;A)\big\}.
$$
%\end{equation}
Then, 
\begin{equation}\label{relbulk}
h_{\mathcal G}(x_0,a,\xi,\Lambda,M) = \lim_{\varepsilon\to 0^+}
\frac{m_{\mathcal G}(a+\xi(\cdot-x_0)+ \frac{1}{2}\Lambda(\cdot-x_0,\cdot-x_0),M;Q(x_0,\varepsilon))}{\varepsilon^N},
\end{equation}
\begin{equation}\label{relsurf}
\psi_{\mathcal G}(x_0,a,\xi,\eta,\nu) = \lim_{\varepsilon\to 0^+}
\frac{m_{\mathcal G}(a+u_{\xi,\eta,\nu}(\cdot-x_0),O;Q_\nu(x_0,\varepsilon))}
{\varepsilon^{N-1}},
\end{equation}
for all $x_0 \in \Omega$, $a \in \mathbb R^d$, $\xi,\eta \in \mathbb R^{d\times N}$,
$\Lambda, M \in \mathbb R^{d\times N\times N}_{sym}$ and $\nu \in \mathbb {S}^{N-1}$, where
$O$ denotes the zero $\mathbb R^{d\times N\times N}$ tensor and $u_{\xi,\eta,\nu}$
is defined by
\begin{equation}\label{jumpf}
u_{\xi,\eta,\nu}(y) := \begin{cases}
\xi \cdot y, & {\rm if } \; y \cdot \nu \geq 0\\
\eta \cdot y, & {\rm if } \; y \cdot \nu < 0,
\end{cases}
\end{equation}
and $\xi - \eta = \zeta \otimes \nu$, for some $\zeta \in \mathbb R^{d\times N}$.
\medskip

We now consider the functional defined in \eqref{Funct}, where the densities $f_0$ and $g_0$ satisfy hypotheses (G1)-(G2), and we let 
$\mathcal F : SD_2(\Omega) \times \mathcal{O}(\Omega) \to [0, +\infty[$ be given by 
\eqref{relaxf}.
\color{black}
We point out that, in the particular case where $h_0$ can be decoupled as the sum
$$h_0(x,a,\xi,\Lambda) = f_0(\Lambda) + f_1(x,a,\xi), \; \forall x \in \Omega,
a \in \mathbb R^d, \xi \in \mathbb R^{d\times N}, 
\Lambda \in \mathbb R^{d\times N\times N}_{sym},$$
and $\psi_0$ does not depend explicitly on $(x,u)$, and does not depend on $\xi$ and $\eta$ separately, but rather on their difference, then the energies $F$ and $F_0$ are equal and the functionals $\mathcal F$ and $\mathcal G$ coincide, provided that the growth conditions from above in (H1)-(H2) are replaced by those of (G1)-(G2).
These more restrictive growth conditions from above will be important to prove the $SBH$-ellipticity property under consideration, (see Proposition \ref{propgrowth}). However, notice that our model is different from the case considered in \cite{Silhavy}, where the result refers only to a purely interfacial energy.

Furthermore, as mentioned in the Introduction, the continuity, with respect to the $W^{1,1}(\Omega;\mathbb R^d)$ convergence of $u_n$ to $u$, of the term involving $f_1$ in $F_0$, allows us to neglect it in the relaxation process, so in what follows we will consider the functional 
$\mathcal F_1: SD_2(\Omega) \times \mathcal{O}(\Omega) \to [0, +\infty[$  given by 
\begin{align}\label{Lrelaxf}
\mathcal F_1 (u,U;A) &:= \inf \Big\{ \liminf_{n\to +\infty} 
\left(\int_A f_0(\nabla^2 u_n(x)) \, dx +
\int_{S(\nabla u_n) \cap A} g_0([\nabla u_n](x), \nu_{\nabla u_n}(x)) 
\, d\mathcal H^{N-1}(x)\right) :\nonumber \\
& \hspace{2,5cm}u_n \to u \; \text{in} \; L^{1}(\Omega; \mathbb R^d), 
\nabla^2 u_n \wsto U \; \text{in}\; 
\mathcal{M} (\Omega; \mathbb R_{sym}^{d\times N\times N})\Big\}.
\end{align}

Hence, adapting \cite[Lemmas 5.1 - 5.3]{FHP}, we can show that $\mathcal F_1$ satisfies the following properties.

\begin{proposition}\label{calFprop} 
Let $f_0$ and $g_0$ be functions satisfying hypotheses (G1)-(G2) and let $\mathcal F_1$ be given by \eqref{Lrelaxf}.
% where $F_0$ is defined in \eqref{Funct}. 
Then
\begin{itemize}
\item[(i)] (lower semi-continuity) for every $(u,U) \in SD_2(\Omega)$, 
$A \in \mathcal O(\Omega)$ and every sequence $\{(u_n,U_n)\} \subset SD_2(\Omega)$ such that $u_n \to u$ in $L^1(\Omega;\mathbb R^d)$ and $U_n \wsto U$ in 
$\mathcal{M} (\Omega; \mathbb R_{s}^{d\times N\times N})$, it follows that
$$\mathcal F_1(u,U;A) \leq \liminf_{n\to +\infty}\mathcal F_1(u_n,U_n;A);$$
\item[(ii)] (locality)  $\mathcal F_1$ is local, that is, for every
$A\in \mathcal O(\Omega)$, if $(u,U), (v,V) \in SD_2(\Omega)$ are such that $u=v$, $U=V$, for $\mathcal L^N$ a.e $x \in A$, then 
$\mathcal F_1(u,U;A) = \mathcal F_1(v,V;A)$;
\item[(iii)] (growth) there exists $C >0$ such that, for every 
$(u,U) \in SD_2(\Omega)$, $A \in \mathcal O(\Omega)$, we have
$$\frac{1}{C} \Big(\|U\|_{L^1(A;\mathbb R_{sym}^{d\times N\times N})} + |D^2u|(A)\Big) \leq \mathcal F_1(u,U;A) \leq 
C \Big(\|U\|_{L^1(A;\mathbb R_{sym}^{d\times N\times N})} +
|D^2u|(A)\Big);$$
\item[(iv)] (Radon measure) for every $(u,U) \in SD_2(\Omega)$, the set function
$\mathcal F_1(u,U;\cdot)$ is the restriction to $\mathcal O(\Omega)$ of a Radon measure.
\end{itemize}
\end{proposition} 

Once again, Proposition \ref{calFprop} ensures that $\mathcal F_1$ satisfies the hypotheses of \cite[Theorem 4.6]{FHP}. Therefore, we may conclude that 
$\mathcal F_1$
admits an integral representation of the form
\begin{align}\label{Fintrep}
\mathcal F_1(u,U; A) &= 
\int_A h_{\mathcal F_1}(x,u(x),\nabla u(x),\nabla^2 u(x), U(x))\, dx \nonumber \\
& + \int_{S(\nabla u)\cap A} \psi_{\mathcal F_1}(x,u(x),(\nabla u)^+(x),
(\nabla u)^-(x),\nu_{\nabla u}(x))\, d \mathcal H^{N-1}(x),
\end{align}
for all $(u,U) \in SD_2(\Omega)$ and all $A \in \mathcal{O}(\Omega)$. In the previous expression,
the relaxed energy density functions 
$h_{\mathcal F_1}: \Omega \times \mathbb R^d \times \mathbb R^{d\times N}\times \mathbb R^{d\times N\times N}_{sym}\times \mathbb R^{d\times N\times N}_{sym} 
\to [0, +\infty[$ and 
$\psi_{\mathcal F_1}: \Omega  \times \mathbb R^d \times \mathbb R^{d\times N}\times \mathbb R^{d\times N} \times \mathbb {S}^{N-1} \to [0, +\infty[ $ are given by \eqref{relbulk} and \eqref{relsurf}, respectively, with $m_{\mathcal G}$ replaced by
%\begin{equation}\label{mfun}
$$m_{\mathcal F_1}(u,U;A) := \inf \big\{\mathcal F_1(v,V;A) : 
(v,V) \in \mathcal U(u,U;A)\big\}
$$
%\end{equation}
and $\mathcal U(u,U;A)$ defined in \eqref{adm}, for 
$(u,U;A) \in SD_2(\Omega) \times \mathcal O(\Omega)$.

We will now show that, in the case under consideration, the integral representation of $\mathcal F_1$ may be simplified to 
\begin{equation}\label{fg}
\mathcal F_1(u,U; A) = \int_A f(\nabla^2 u(x), U(x))\, dx + 
\int_{S(\nabla u)\cap A} g((\nabla u)^+(x),(\nabla u)^-(x),\nu_{\nabla u}(x))
\, d \mathcal H^{N-1}(x),
\end{equation}
for certain functions $f$ and $g$.

\begin{proposition}\label{simprelax}
Let $f_0$ and $g_0$ be functions satisfying hypotheses (G1)-(G2) and let 
$\mathcal F_1$ be given by \eqref{Lrelaxf}. Then, for every $x_0 \in \Omega$, 
$a \in \mathbb R^d$, $\xi,\eta \in \mathbb R^{d\times N}$,
$\Lambda, M \in \mathbb R^{d\times N\times N}_{sym}$ and 
$\nu \in \mathbb {S}^{N-1}$, we have
\begin{equation}\label{h=h}
h_{\mathcal F_1}(x_0,a,\xi,\Lambda,M) = h_{\mathcal F_1}(0,0,0,\Lambda,M)
\end{equation}
and
\begin{equation}\label{psi=psi}
\psi_{\mathcal F_1}(x_0,a,\xi,\eta,\nu) 
= \psi_{\mathcal F_1}(0,0,\xi-\eta,0,\nu),
\end{equation}
where $h_{\mathcal F_1}$ and $\psi_{\mathcal F_1}$ are the functions appearing in \eqref{Fintrep}.
\end{proposition}

\begin{remark}\label{rmk3.4}
Due to \eqref{h=h} and \eqref{psi=psi}, we conclude that  
\eqref{fg} holds for the functions $f$ and $g$ given by 
$f(\Lambda,M) : =  h_{\mathcal F_1}(0,0,0,\Lambda,M)$ and
$g(\xi,\eta,\nu) := \psi_{\mathcal F_1}(0,0,\xi-\eta,0,\nu)$. Thus we observe that, as in the case of the density $g_0$, $g$ does not depend on $\xi$ and $\eta$ independently, but only on the difference $\xi - \eta$.
\end{remark}

\begin{proof}[Proof of Proposition \ref{simprelax}]

We begin by proving equality \eqref{h=h}. 

Given $a \in \mathbb R^d$ and
$\xi \in \mathbb R^{d\times N}$ we define the affine function
$u_{a,\xi}(x) := a + \xi \cdot x$, for $x \in \Omega$.

Observe that our assumptions on $f_0$ and $g_0$ entail that 

\begin{align}\label{F1=F1}
	F_1(u;A)= F_1(u+u_{a,\xi};A),
\end{align} 

for every $a \in \mathbb R^d$, $\xi \in \mathbb R^{d\times N}$ and $u \in SBH(\Omega;\mathbb R^d)$, 
    %and $U \in L^1(\Omega;\mathbb R^{d\times N \times N}_{sym})$, 
where 
\begin{align}\label{F1}
F_1(u;A):= \int_A f_0(\nabla^2 u(x)) \, dx +
\int_{S(\nabla u) \cap A} g_0([\nabla u](x),\nu_{\nabla u}(x)) \,
d \mathcal H^{N-1}(x),
\end{align}

for every $u \in SBH(\Omega;\mathbb R^d)$ and $A \in \mathcal O(\Omega)$.

This property of translation invariance with respect to affine functions  is clearly inherited by ${\mathcal F}_1$, since  $\{u_n\}$ is an admissible sequence for $\mathcal F_1(u,U;A)$ if, and only if, $\{ u_n+ u_{a,\xi}\}$ is admissible for 
$\mathcal F_1(u+u_{a,\xi},U;A)$ and we are taking into account \eqref{F1=F1}. 
From this we conclude that
$$m_{\mathcal F_1}(u,U;A) = m_{\mathcal F_1}(u+u_{a,\xi},U;A),$$
which clearly ensures that   
$$h_{\mathcal F_1}(x_0,a,\xi,\Lambda,M) = h_{\mathcal F_1}(x_0,0,0,\Lambda,M),$$
for every  $x_0\in\Omega$, $(a,\xi)\in \mathbb R^d \times \mathbb R^{d\times N}$ and $\Lambda, M \in \mathbb R^{d \times N \times N}_{sym}$.

To conclude that 
$$h_{\mathcal F_1}(x_0,0,0,\Lambda,M) = h_{\mathcal F_1}(0,0,0,\Lambda,M),$$
for every  $x_0\in\Omega$ and 
$\lambda, M \in \mathbb R^{d \times N \times N}_{sym}$, 
we use a change of variables argument. Indeed, it suffices to notice that 
$(u,U)$ is an admissible couple  for 
$m_{\mathcal F_1}
\left(\frac{1}{2}\Lambda(\cdot-x_0,\cdot-x_0),M;Q(x_0,\varepsilon)\right)$ 
if, and only if, the pair $(v,V)$, with $v(y):=u(x_0+y)$ and $V(y):=U(x_0+y)$, 
is admissible for 
$ m_{\mathcal F_1}
\left(\frac{1}{2}\Lambda(\cdot,\cdot), M;Q(0,\varepsilon)\right)$, 
taking also into account that 
$F_1(u_n;Q(x_0,\varepsilon))= F_1(v_n; Q(0,\varepsilon))$
whenever the admissible sequences $\{u_n\}$ and $\{v_n\}$ are related in the same way, that is, $v_n(y)=u_n(x_0+y)$. 

We now proceed with the proof of equality \eqref{psi=psi}.
 This amounts to showing that
$$m_{\mathcal F_1}
\big(a + u_{\xi,\eta,\nu}(\cdot -x_0),O;Q_\nu(x_0,\varepsilon)\big) =
m_{\mathcal F_1}
\big(u_{\xi-\eta,0,\nu},O;Q_\nu(0,\varepsilon)\big),$$ 
where $u_{\xi,\eta,\nu}$ is given in \eqref{jumpf};
as above, 
this relies on a change of variables argument using translations. 
In fact, if a pair $(v, V)$ is admissible for
$m_{\mathcal F_1}\big(a + u_{\xi,\eta,\nu}(\cdot -x_0),O;Q_\nu(x_0,\varepsilon)\big)$, then the pair
$(w,W)$, where, for $y \in Q_\nu(0,\varepsilon)$,  $w$  and $W$ are defined by $w(y) := v(x_0 +y) - a - \eta \cdot y$, $W(y) := V(x_0+y)$, 
is admissible for
$m_{\mathcal F_1}\big(u_{\xi-\eta,0,\nu},O;Q_\nu(0,\varepsilon)\big)$, with the  energies satisfying 
\begin{equation}\label{equal calF}
\mathcal F_1(w,W;Q_\nu(0,\varepsilon)= \mathcal F_1(v,V;Q_\nu(x_0,\varepsilon)). \end{equation}
Likewise, departing from an admissible pair for  
$m_{\mathcal F_1}\big(u_{\xi-\eta,0,\nu},O;Q_\nu(0,\varepsilon)\big)$, a similar
construction yields an admissible pair for 
$m_{\mathcal F_1}\big(a + u_{\xi,\eta,\nu}(\cdot -x_0),O;Q_\nu(x_0,\varepsilon)\big)$ satisfying \eqref{equal calF}.
\end{proof}

\begin{lemma}\label{calFjump}
Let $f_0$ and $g_0$ be functions satisfying hypotheses (G1)-(G2) and let $\mathcal F$ be given by \eqref{Lrelaxf}. Consider the function $f$ defined in Remark \ref{rmk3.4}. 
Then
\begin{align}\label{f=0}f(\nabla^2 u_{\xi,\eta,\nu},O) = f(O,O) = 0,
	\end{align}
where $u_{\xi,\eta,\nu}$ was defined in \eqref{jumpf}.
\end{lemma}
\begin{proof}
It follows immediately from Proposition \ref{calFprop} (iii) that, for any 
$A \in \mathcal O(\Omega)$,
$$\mathcal F(u_{\xi,\eta,\nu},O;A) \leq C |D^2u_{\xi,\eta,\nu}|(A),$$
where 
$D^2u_{\xi,\eta,\nu}=D^s(\nabla u_{\xi,\eta,\nu})$ since 
$\nabla^2 u_{\xi,\eta, \nu}=O$. 
Consequently, taking into account that $D^2u_{\xi,\eta,\nu}$ is 
concentrated on $\{x \cdot \nu=0\}$, whose $\mathcal L^N$ measure is zero, \eqref{f=0} follows trivially by \eqref{fg}.
\end{proof}

\begin{proposition}
    \label{propgrowth}
    Let $f_0$ and $g_0$ be functions that satisfy hypotheses (G1)-(G2) and let $\mathcal F_1$ be given by \eqref{Lrelaxf}. Let $f$ and $g$ be the densities in \eqref{fg}. Then, there exists $C>0$ such that
    \begin{align}
    \label{fcoerc}
   \frac{1}{C}(|\Lambda|+ |M|)\leq f(\Lambda, M)\leq C(|\Lambda|+ |M|),\\
   \frac{1}{C}|\xi-\eta|\leq g(\xi,\eta,\nu)\leq C(|\xi-\eta|), \label{gcoerc}
    \end{align}
    for every $\Lambda, M \in \mathbb R^{d\times N\times N}_{sym}$, $\xi, \eta \in \mathbb R^d$ and $\nu \in \mathbb S^{N-1}.$
\end{proposition}
    \begin{proof}[Proof]
    In view of Proposition \ref{calFprop} (iii) and \eqref{fg}
    it follows that 
\begin{align*}
&\frac{1}{C} \Big(\|U\|_{L^1(A;\mathbb R_{sym}^{d\times N\times N})} + |D^2u|(A)\Big) \\
&\hspace{1cm}\leq \int_A f(\nabla^2 u(x), U(x))\, dx
+ \int_{S_{\nabla u}\cap A} g((\nabla u)^+(x), (\nabla u)^-(x), \nu_{\nabla u}(x)) \, d \mathcal H^{N-1}(x) \\
&\hspace{1cm}\leq 
C \Big(\|U\|_{L^1(A;\mathbb R_{sym}^{d\times N\times N})} +
|D^2u|(A)\Big),
\end{align*}
for every $A \in \mathcal O(\Omega)$  and $(u,U) \in SD_2(\Omega)$.
In particular, letting  $U= M \in \mathbb R^{d\times N \times N}_{sym}$,
$u(x) = \frac{1}{2}\Lambda (x,x)$,  with $\Lambda \in \mathbb R^{d\times N \times N}_{sym}$, we conclude the growth conditions stated in \eqref{fcoerc}. 

On the other hand,  by taking $U=0$, $u= u_{\xi, \eta,\nu}$ given in \eqref{jumpf}, referring also to Lemma \ref{calFjump},  we obtain \eqref{gcoerc}.
    \end{proof}

\begin{theorem}\label{prop4.12}
Let $F_1$ and $\mathcal F_1$ be the functionals defined in 
\eqref{F1} and \eqref{Lrelaxf}, respectively,  where $f_0$ and $g_0$ 
satisfy (G1)-(G2). 
If the  functions $f$ and $g$ in \eqref{fg} are continuous, then  $g$ is 
$SBH$-elliptic, 
that is, for every $\nu \in \mathbb S^{N-1}$, 
$\xi,\eta \in \mathbb R^{d \times N}$ such that  
$\xi - \eta = \zeta \otimes \nu$, for some $\zeta \in \mathbb R^d$,
\begin{align}\label{SBHtoprove}
g(\xi,\eta,\nu) &=
\int_{Q_\nu \cap \{x \cdot \nu =0\}} g(\xi,\eta, \nu) 
\, d \mathcal H^{N-1}(x) \nonumber \\ 
&\leq \liminf_{n\to +\infty} \int_{Q_\nu \cap \{x \cdot \nu =0\}}
g((\nabla u_n)^+(x),(\nabla u_n)^-(x), 
\nu_{\nabla u_n}(x)) \,d {\mathcal H}^{N-1}(x), 
\end{align}
whenever and $\{u_n\} \subset SBH(Q_\nu;\mathbb R^d)$ is such that
$u_n \to u_{\xi,\eta,\nu}$ in $W^{1,1}(Q_\nu;\mathbb R^d)$, $|\nabla^2 u_n|\to 0$ in $L^1(Q_\nu)$ and $u_n= u_{\xi,\eta,\nu}$  on $\partial Q_\nu$, where  $u_{\xi,\eta,\nu}$ is given in \eqref{jumpf}. 
\end{theorem}
\begin{proof}[Proof]
Let $\{u_n\}$ be a sequence satisfying the conditions of the statement and such that the right-hand side of \eqref{SBHtoprove} is finite, otherwise the result is trivial.
Without loss of generality, up to passing to a subsequence, we can assume that the liminf is a limit.

In particular, the coercivity condition \eqref{gcoerc} entails that
\begin{align}\label{gb}
\lim_n |(D^2)^s u_n|(Q_\nu \cap \{x\cdot \nu =0\})< +\infty.
\end{align}
Then, by Lemma \ref{calFjump}, applying \eqref{fg} twice, to identify first 
$\mathcal F_1(u_{\xi,\eta,\nu}, O; Q_\nu)$ and then 
$\mathcal F_1(u_n,O; Q_\nu)$, and using the lower semi-continuity property of $\mathcal F_1$ stated in Proposition \ref{calFprop}(i), we have

\begin{align*}
	\int_{Q_\nu \cap \{x \cdot \nu =0\}}\
    g(\xi,\eta, \nu) \, d \mathcal H^{N-1}(x)
	& =\int_{Q_\nu}f(O, O) \, dx + 
	\int_{Q_\nu \cap \{x \cdot \nu =0\}}
    g (\xi,\eta, \nu) \, d \mathcal H^{N-1}(x)\\
	&=\mathcal F_1(u_{\xi,\eta,\nu}, O; Q_\nu)\\
	&\leq \liminf_{n\to +\infty}\mathcal F_1(u_n,O; Q_\nu)\\
    &= \liminf_{n\to +\infty}\left(\int_{Q_\nu} 
	f(\nabla^2 u_n(x),O) \, dx \right. \\
    & \left. \hspace{2cm} + \int_{Q_\nu \cap \{x \cdot \nu =0\}}
	g((\nabla u_n)^+(x),(\nabla u_n)^-(x), \nu_{\nabla u_n}(x)) \, d {\mathcal H}^{N-1}(x)\right)\\
	&= \liminf_{n\to +\infty} \int_{Q_\nu \cap \{x \cdot \nu =0\}}
	g((\nabla u_n)^+(x),(\nabla u_n)^-(x), \nu_{\nabla u_n}(x)) \, d {\mathcal H}^{N-1}(x),
\end{align*}
where we used the Vitali-Lebesgue convergence theorem in the last equality, taking into account  the continuity of $f$ and that, by Proposition \ref{calFprop} (iii),
the following bound holds
$$\int_{Q_\nu} f(\nabla^2 u_n(x),O) \, dx \leq \mathcal F_1(u_n,O; Q_\nu)
\leq C |D^2u_n|(Q_\nu),$$
where the right-hand side is uniformly bounded in view of \eqref{gb} and the fact that $|\nabla^2 u_n| \to 0$, in $L^1(Q_\nu)$, as $n \to +\infty$.
\end{proof}

\bigskip
\subsection*{Acknowledgements}

The research of Ana Cristina Barroso was partially supported by National Funding from FCT - Funda\c c\~ao para a Ci\^encia e a Tecnologia through project 
UID/04561/2025. 

%\noindent https://doi.org/10.54499/UIDB/04561/2020.

The research of Jos\'{e} Matias was funded by FCT/Portugal through project UIDB/04459/2020 with DOI identifier 10-54499/UIDP/04459/2020.

Elvira Zappale acknowledges the support of Piano Nazionale di Ripresa e Resilienza (PNRR) - Missione 4 ``Istruzione e Ricerca''- Componente C2 Investimento 1.1, "Fondo per il Programma Nazionale di Ricerca e
Progetti di Rilevante Interesse Nazionale (PRIN)" - Decreto Direttoriale n. 104 del 2 febbraio 2022 - CUP 853D23009360006. She is a member of the Gruppo Nazionale per l'Analisi Matematica, la Probabilit\`a e le loro Applicazioni (GNAMPA) of the Istituto Nazionale di Alta Matematica ``F.~Severi'' (INdAM). 
She also acknowledges partial funding from the GNAMPA Project 2024 \emph{Composite materials and microstructures}, coordinated by M. Amar.
The work of EZ is also supported by Sapienza - University of Rome through the projects Progetti di ricerca medi, (2021), coordinator  S. Carillo e Progetti di ricerca piccoli,  (2022), coordinator E. Zappale.

\bibliographystyle{plain}

\end{document}